\documentclass[a4paper,abstracton]{scrartcl}
\usepackage[american]{babel}

\usepackage{graphicx}
\usepackage[matrix,arrow]{xy}

\usepackage{amsmath,amssymb,amsthm}
\theoremstyle{definition}

\swapnumbers
\theoremstyle{plain}
\newtheorem{theorem}{Theorem}[section]
\newtheorem{lemma}[theorem]{Lemma}
\newtheorem{prop}[theorem]{Proposition}
\newtheorem{cor}[theorem]{Corollary}
\theoremstyle{definition}
\newtheorem{definition}[theorem]{Definition}
\newtheorem{algorithm}[theorem]{Algorithm}

\makeatletter
\def\@fnsymbol#1{\ensuremath{\ifcase#1\or *\or \dagger
   \or \sharp \or \flat
   \or \circledast \or \diamond
   \or \ddagger \else\@ctrerr\fi}}
\makeatother

\newcommand{\NP}{\mathrm{NP}}
\newcommand{\PPAD}{\mathrm{PPAD}}
\newcommand{\coNP}{\mathrm{coNP}}

\usepackage{algorithmic}

\def\R{\mathbb{R}}

\def\id{\mathrm{id}}
\let\Phi\phi
\def\toOver#1{\xrightarrow{#1}}
\def\toPhi{\toOver{\Phi}}
\DeclareMathOperator{\LCP}{LCP}
\DeclareMathOperator{\PLCP}{P-LCP}

\def\subcube#1#2{{#1}[{#2}]}

\begin{document}

\title{Pivoting in Linear Complementarity:\\Two Polynomial-Time Cases}

\author{%
Jan Foniok%
	\thanks{ETH Zurich, Institute for Operations Research, 8092 Zurich, Switzerland}
	\thanks{\texttt{foniok@math.ethz.ch}}
\and
Komei Fukuda%
	\footnotemark[1]
	\thanks{ETH Zurich, Institute of Theoretical Computer Science, 8092 Zurich, Switzerland}
	\thanks{\texttt{fukuda@ifor.math.ethz.ch}}
\and
Bernd G{\"a}rtner%
	\footnotemark[3]
	\thanks{\texttt{gaertner@inf.ethz.ch}}
\and
Hans-Jakob L\"uthi%
	\footnotemark[1]
	\thanks{\texttt{luethi@ifor.math.ethz.ch}}
}

\maketitle

\begin{abstract}
  We study the behavior of simple
  principal pivoting methods for the P-matrix linear complementarity
  problem (P-LCP).  We solve an open problem of Morris by
  showing that Murty's least-index pivot rule (under any fixed index
  order) leads to a quadratic number of iterations on Morris's highly
  cyclic P-LCP examples. We then show that on K-matrix LCP instances,
  \emph{all} pivot rules require only a linear number of iterations. 
  As the main tool, we employ \emph{unique-sink orientations} of cubes, a
  useful combinatorial abstraction of the P-LCP.
\end{abstract}

\section{Introduction}

The third author of this paper still vividly recollects his visit to
Victor Klee at the University of Washington (Seattle) in August 2000.

\begin{quotation}
  We had a memorable drive in Vic's car from the hotel to the math
  department. Vic skipped all the small talk and immediately asked:
  ``Do you think that the simplex method is polynomial?'' I still
  remember how awkward I felt to be asked this question by the very
  person who had provided the first and seminal insights into it.
  My (then as now quite irrelevant) opinion shall remain between Vic
  and me forever.
\end{quotation}

The seminal work referred to above is of course the 1972 paper of Klee
and Minty entitled ``How good is the simplex algorithm?''
\cite{KleMin:How-good}. It deals with the number of iterations that
may be required in the worst case to solve a linear program (LP) by
Dantzig's \emph{simplex method} \cite{Dan:Linear}, which
was at that time the only available practical method. Klee and Minty
exhibited a family of LPs (now known as the \emph{Klee-Minty cubes}) for
which the number of iterations is \emph{exponential} in the number of
variables and constraints.

Dantzig's specific \emph{pivot rule} for selecting the next step is
just one of many conceivable rules.  The question left open
by the work of Klee and Minty is whether there is some \emph{other}
rule that provably leads to a small number of iterations. This
question is as open today as it was in 1972.

Nowadays, the simplex method is no longer the only available method to
solve LPs. In particular, there \emph{are} proven polynomial-time
algorithms for solving LPs---Khachiyan's \emph{ellipsoid method}
\cite{Kha} and  Karmarkar's \emph{interior point method}
\cite{Kar}---that are based on techniques developed originally for
nonlinear optimization. It is still unknown, though, whether there is
a \emph{strongly} polynomial-time algorithm for solving LPs, that is, an algorithm
for which the number of arithmetic operations does not depend on the
bit lengths of the input numbers. 

\paragraph{The P-matrix linear complementarity problem.}
In this paper we are concerned with pivoting methods for the
\emph{P-matrix linear complementarity problem} (P-LCP), a prominent
problem for which neither polynomial-time algorithms nor hardness
results are available. 

In general, an LCP is given by a matrix $M\in\R^{n\times n}$ and a
vector $q\in\R^n$, and the problem is to find vectors $w,z\in\R^n$
such that
\begin{equation}\label{eq:lcpdef}
w-Mz = q, \quad w,z\geq 0, \quad w^Tz = 0.
\end{equation}
It is $\NP$-complete to decide whether such vectors exist
\cite{Chung:Hardness,KojMegNom:A-unified}. 
However, in a P-LCP, $M$~is a P-matrix (meaning that
all principal minors are positive), and then there are unique  
solution vectors $w,z$ for every right-hand side~$q$~\cite{STW}.
The problem of efficiently finding them is unsolved, though.

Megiddo has shown that the P-LCP is unlikely to be NP-hard. For
this, he considers the following (more general) variant: given
$M$ and $q$, either find $w,z$ that satisfy (\ref{eq:lcpdef}), or
exhibit a nonpositive principal minor of $M$. NP-hardness of the latter
variant would imply that $\NP=\coNP$~\cite{Meg:A-Note-on-the-Complexity}.

There is a different notion of hardness that might apply to the P-LCP as
a member of the complexity class~$\PPAD$~\cite{PPAD}.  This class has
complete problems, and no polynomial-time algorithm is known for any of
them. It has recently been shown that the famous problem of finding a
Nash equilibrium in a bimatrix game is $\PPAD$-complete~\cite{Nash}.
Some researchers consider this to be an explanation why (despite many
efforts) no polynomial-time algorithm has been found so far.
Incidentally, this is the \emph{second} problem for which Megiddo
showed in his technical report~\cite{Meg:A-Note-on-the-Complexity}
that it is unlikely to be NP-hard. However, it is currently not known whether
the P-LCP is also PPAD-complete.

There are various algorithms for solving P-LCPs, among them the
classic method by Lemke \cite{Lemke:Recent-Results} as well as
interior point approaches
\cite{KojMegYe:An-interior,KojMegNom:A-unified}. Still, there is no
known polynomial-time algorithm for P-LCPs. For example, the
complexity of interior point algorithms depends linearly on a matrix
parameter $\kappa$ \cite{KojMegNom:A-unified} that is not bounded for
P-matrices.

In this work, we focus on combinatorial methods for P-LCPs and in
particular on \emph{simple principal pivoting methods} that share
their essential idea with the simplex method.

\paragraph{Simple principal pivoting methods.}
Let $B\subseteq\{1,2,\ldots,n\}$, and let $A_B$ be the $n\times n$ matrix whose $i$th
column is the $i$th column of~$-M$ if $i\in B$, and the $i$th column of
the $n\times n$ identity matrix~$I_n$ otherwise.  If $M$ is a P-matrix,
then $A_B$ is invertible for every set $B$. We call $B$ a \emph{basis}.
If $A_B^{-1}q\geq 0$, we have discovered the solution: let 
\begin{equation}\label{eq:lcpsol}
w_i := \begin{cases}
0 & \text{if $i\in B$} \\
(A_B^{-1}q)_i & \text{if $i\notin B$}
\end{cases}, \qquad 
 z_i := \begin{cases}
(A_B^{-1}q)_i & \text{if $i\in B$}\\
0 & \text{if $i\notin B$}
\end{cases}.
\end{equation}

If on the other hand $A_B^{-1}q\ngeq 0$, then $w$ and $z$ defined above
satisfy $w-Mz=q$ and $w^Tz=0$,
but $w,z\geq 0$ fails. In \emph{principal pivoting}, one tries to
improve the situation by replacing the basis $B$ with the symmetric
difference $B\oplus C$, where $C$ is some nonempty subset of the ``bad
indices'' $\{i:(A_B^{-1}q)_i<0\}$. The greedy choice is to let $C$ 
be the set of all bad indices in every step.  For some P-LCPs, however, this algorithm
cycles and never terminates~\cite{Mur:Linear}.\footnote{For $n=3$,
one such example is the P-LCP (\ref{eq:MorrisMq}) whose digraph model (see
next paragraph) appears in Figure \ref{fig:uso}. Cycling occurs for any 
of the three start bases just below the top vertex in Figure~\ref{fig:uso}.} 
In \emph{simple principal pivoting}, $C$
is of size one, and a \emph{pivot rule} is employed to select the bad
index. Simple principal pivoting methods are sometimes called
\emph{Bard-type} methods and were first studied by
Zoutendijk~\cite{Zou:Methods} and Bard~\cite{Bar:An-eclectic}.  For a
detailed exposition on simple principal pivoting methods 
see, for instance, \cite[Section~4.2]{CotPanSto:LCP}
or~\cite[Chapter~4]{Mur:Linear}.

\paragraph{The digraph model and unique sink orientations.}
There is a natural digraph model behind principal pivoting, first
studied by Stickney and Watson \cite{StiWat:Digraph-models}. The graph's
vertices are all bases~$B$, with a directed edge going from $B$ to
$B'$ if $B\oplus B'=\{i\}$ and $(A_B^{-1}q)_i<0$.\footnote{We require
  $q$ to be nondegenerate, meaning that $(A_B^{-1}q)_i\neq 0$ for all
  $B$ and $i$. This is no serious restriction since a nondegenerate
  problem can always be obtained from a degenerate one by a symbolic
  perturbation of $q$.} The underlying undirected graph is the graph
of the $n$-dimensional cube. A principal pivot step considers the cube
that is spanned by the outgoing edges at the current vertex, and it
jumps to the antipodal vertex of some subcube of it. A simple
principal pivot step can be interpreted as the traversal of an
outgoing edge, and this is the analogy with the simplex method for
linear programming.

The above digraph has the following specific property: every nonempty
subcube (including the whole cube) has a unique sink, and the global
sink corresponds to the solution of the P-LCP
\cite{StiWat:Digraph-models}. In the terminology of Szab{\'o} and
Welzl \cite{SzaWel:UniSink}, we are dealing with a \emph{unique-sink
orientation} (USO). 

The goal of this paper is to deepen the understanding of the USO
description of simple principal pivoting methods for the P-LCP. On a
general level, we want to understand whether and how algebraic
properties of a P-LCP (which are well studied) translate to
combinatorial properties of the corresponding USO (which are much less studied), and
what the algorithmic implications of such translations are. The
principal motivation behind this research is the continuing quest for
a polynomial-time P-LCP algorithm. In this paper, the concept of a USO
serves as the main tool to obtain two more positive results for simple
principal pivoting. 

We believe that this combinatorial approach has
some untapped potential for the theory of (simple)
principal pivoting methods for LCPs.

\paragraph{The (randomized) method of Murty and the Morris orientations.} 
The simple principal pivoting method of Murty (also known as the
\emph{least-index} rule) \cite{Mur:Note-on-a-Bard-type} works for all
P-LCP instances, but it may take exponentially many iterations in the
worst case \cite{Murty:slow}. As a possible remedy, researchers have
considered \emph{randomized} methods. The two most prominent ones are
{\sc RandomizedMurty} (which is just the least-index rule, after
randomly reshuffling the indices at the beginning), and {\sc
RandomEdge} (which performs a purely random walk in the USO). 
However, Morris found a family of P-LCP instances (we call
their underlying digraphs the \emph{Morris orientations}) on which
{\sc RandomEdge} spends a long time running in cycles and thus
performs much worse than even the exhaustive search
algorithm~\cite{Mor:Randomized-pivot}.

For {\sc RandomizedMurty} there are as yet no such
negative results on P-LCP instances. In particular, on Murty's
worst-case example, this algorithm takes an expected \emph{linear}
number of steps \cite{FukNam:On-extremal-behaviors}; the expected number
of steps becomes
quadratic if we allow arbitrary start vertices~\cite{BP}.

On general P-LCP instances, the expected number of iterations taken
by {\sc RandomizedMurty} is \emph{subexponential} if the underlying digraph is
acyclic~\cite{Gar:The-random-facet}, but it is unknown whether this also
holds under the presence of directed cycles (as they appear in the Morris
orientations, for example).

We prove that not only the randomized variant, but actually \emph{any}
variant of the least-index rule obtained by some initial reshuffling
of indices leads to a \emph{quadratic} number of iterations on the Morris
orientations. In particular, this ``kills'' another family of potentially bad
instances for the {\sc RandomizedMurty} rule.

\paragraph{K-matrix linear complementarity problems.}
A \emph{K-matrix} is a P-matrix so that all off-diagonal elements are non-positive.
LCPs with K-matrices (K-LCPs) appear for example in free boundary problems
\cite[Section~5.1]{CotPanSto:LCP} and in American put option pricing
\cite{options1,options2}.
It is known that every K-LCP instance can be solved
in polynomial time (even by a simple principal
pivoting method~\cite{Cha:A-special,Sai:A-note}, see
also~\cite[Section~4.7]{CotPanSto:LCP}), but we prove something stronger:
\emph{every} simple principal pivoting method takes only a linear number
of iterations. We obtain this result by showing that in
K-LCP-induced USOs, \emph{all} directed paths are short. Our approach
is to extract combinatorial structure from the algebraic properties of
K-matrices. Subsequently, we use the distilled combinatorial
information to get our structural results. 

\paragraph{LCPs with sufficient matrices and the criss-cross method.}
Let us step back and take a broader view that brings together
LP, the P-LCP, pivoting methods, and computational complexity: linear
complementarity problems with \emph{sufficient} matrices generalize
both LP and the P-LCP while still being amenable to pivoting approaches.
The \emph{criss-cross
  method}~\cite{FukNamTam:EP-theorems,FukTer:Linear} may be considered
an extension of Murty's algorithm.  In fact, the results
of~\cite{FukNamTam:EP-theorems} make LCPs with sufficient matrices the
largest known class of LCPs for which Megiddo's techniques
of~\cite{Meg:A-Note-on-the-Complexity} still apply, so that
NP-hardness is unlikely. Determining the complexity of sufficient
matrix LCP remains a tough challenge for the future.

\section{Unique-sink orientations}
\label{sec:uso}
Now we formally introduce the digraph model behind P-LCPs and show
some of its basic properties. The model
was first described by Stickney and Watson in 1978
\cite{StiWat:Digraph-models}.

We use the following notation.  Let
$[n]:=\{1,2,\ldots,n\}$.  For a bit vector $v\in\{0,1\}^n$ and
$I\subseteq [n]$, let $v\oplus I$ be the element of~$\{0,1\}^n$
defined by
\[(v\oplus I)_j :=
	\begin{cases}
	1-v_j&\text{if $j\in I$,}\\
	v_j&\text{if $j\notin I$.}
	\end{cases}
\]
This means that $v\oplus I$ is obtained from $v$ by flipping all
entries with coordinates in~$I$. Instead of $v\oplus\{i\}$ we simply
write $v\oplus i$.

Under this notation, the (undirected) \emph{$n$-cube} is the graph
$G=(V,E)$ with
\[
V := \{0,1\}^n, \quad
E := \{\{v, v\oplus i\} : v\in V,\ i\in[n]\}.
\] 

A \emph{subcube} of~$G$ is a subgraph $G'=(V',E')$ of~$G$ where
$V'=\{v\oplus I:I\subseteq C\}$ for some vertex $v$ and some set
$C\subseteq[n]$, and $E'=E\cap\binom{V'}{2}$.

\begin{definition}
  \label{def:uso}
  Let $\Phi$ be an orientation of the $n$-cube (a digraph with
  underlying undirected graph $G$). We call $\Phi$ a \emph{unique-sink
    orientation} (USO) if every nonempty subcube has a unique sink in
  $\Phi$.
\end{definition}

Figure~\ref{fig:uso} depicts a USO of the 3-cube.  The conditions in
Definition \ref{def:uso} require the orientation to have a unique
global sink, but in addition, all proper nonempty subcubes must have
unique sinks as well. In Figure~\ref{fig:uso} there are six
2-dimensional subcubes, twelve 1-dimensional subcubes (edges) and
eight 0-dimensional subcubes (vertices). For edges and vertices, the
unique-sink condition is trivial. The figure also shows that USOs may
contain directed cycles.

\begin{figure}[!ht]
\begin{center}
\includegraphics{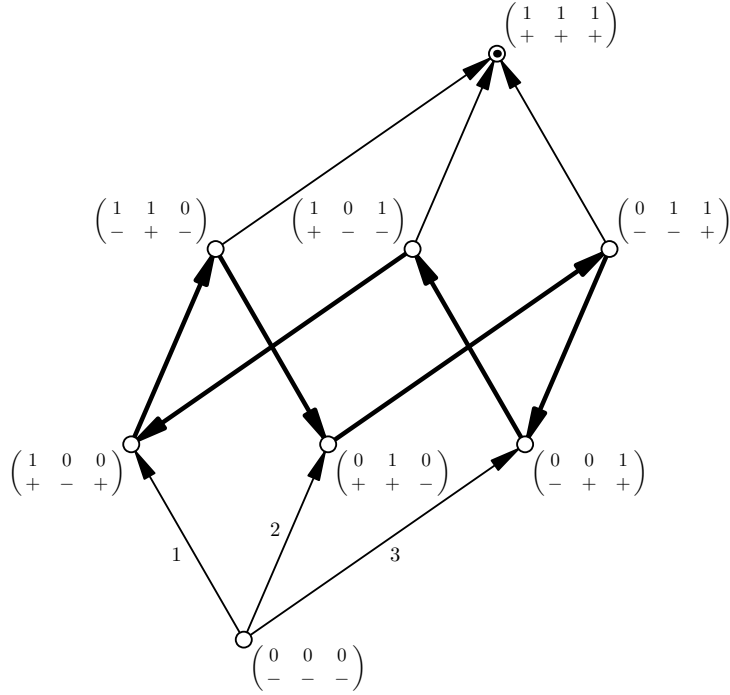}
\end{center}
\caption{A USO of the 3-cube; for each vertex, the orientation of
  incident edges is encoded with a sign vector in $\{-,+\}^n$
  ($-$~for an outgoing edge; $+$~for an incoming edge). The unique global
  sink and a directed cycle are highlighted.
  \label{fig:uso}}
\end{figure}

If $V'$ is the vertex set of a subcube (as defined above),
then the directed subgraph of~$\phi$ induced by~$V'$ is denoted
by~$\subcube{\phi}{V'}$.

It is convenient to view the orientation of the $n$-cube~$G$ as a mapping $\Phi:
\{0,1\}^n\to\{-,+\}^n$, where $\Phi(v)_i = -$ if the edge between $v$
and~$v\oplus i$ is oriented towards $v\oplus i$, and $\Phi(v)_i = +$
if it is oriented towards~$v$.  To encode a vertex~$v$ along with the orientations of
its incident edges, we can then use a \emph{configuration} of $n$~bits and
$n$~signs, where the $i$th sign is $\Phi(v)_i$ (see Figure~\ref{fig:uso}).
If an orientation $\Phi$ of $G$ contains the directed edge $(v, v\oplus
i)$, we write $v\toPhi v\oplus i$, or simply $v\to v\oplus i$
if $\Phi$~is clear from the context.

Let $\Phi$ be an orientation of the $n$-cube and let $F\subseteq[n]$. Then
$\Phi^{(F)}$~is the orientation of the $n$-cube obtained by reversing
all edges in coordinates contained in~$F$; formally
\[v \toOver{\Phi^{(F)}} v\oplus i \quad :\Leftrightarrow \quad
	\begin{cases}
	v \toPhi v\oplus i &\text{if $i\notin F$},\\
	v\oplus i \toPhi v &\text{if $i\in F$}.
	\end{cases}
\]

\begin{prop}[\cite{SzaWel:UniSink}]
If $\Phi$~is a USO, then $\Phi^{(F)}$ is also a USO for an arbitrary
subset~$F\subseteq[n]$.
\qed
\end{prop}

This proposition can be proved by induction on the size of~$F$; it
essentially suffices to show that under $\Phi^{(\{i\})}$, the whole cube
still has a unique sink.

Now we can derive the
\emph{unique-completion property}.  Its meaning is that we can prescribe
the bits for some coordinates and the signs for the remaining coordinates,
and in a USO there will always be a unique vertex that satisfies the
prescription.  In particular, it follows that any USO (and any of its
subcubes) also has a \emph{unique source}.

In fact, the unique-completion property characterizes unique-sink
orientations.

\begin{lemma}
\label{lem:ucp}
An orientation~$\Phi$ of the $n$-cube~$G$ is a
unique-sink orientation if and only if for every partition $[n]=A\cup B$
and every pair of maps $\alpha:A\to\{0,1\}$, $\beta:B\to\{-,+\}$
there exists a unique vertex~$v$ such that
\begin{subequations}
\label{eq:ucp}
\begin{alignat}{2}
v_i &= \alpha(i) &\quad&\text{for all $i\in A$,}\\
\Phi(v)_i &= \beta(i) &&\text{for all $i\in B$.}
\end{alignat}
\end{subequations}
\end{lemma}

\begin{proof}
  First suppose that $\Phi$~is a USO.  Let $U:=\{u: u_i=\alpha(i) \text{
  for all $i\in A$}\}$ and consider the subcube~$\subcube{\phi}{U}$.
  Let $F:=\{i\in B: \beta(i)=-\}$. As we have just noted, $\Phi^{(F)}$~is
  a unique-sink orientation, and thus the subcube~$\subcube{\phi}{U}$ has a
  unique sink~$v$ with respect to~$\Phi^{(F)}$. This vertex~$v$ is the
  only vertex that satisfies~(\ref{eq:ucp}a) and (\ref{eq:ucp}b).

  Conversely, let $\Phi$ be an orientation of the $n$-cube which
  satisfies the unique-completion property and let $U:=\{u\oplus
  I: I\subseteq C\}$. Set $B:=C$ and $A:=[n]\setminus C$, and let
  $\alpha(i):=u_i$ for $i\in A$ and $\beta(i):=+$ for $i\in B$. Then
  the unique vertex~$v$ satisfying~(\ref{eq:ucp}a) and (\ref{eq:ucp}b)
  is the unique sink of the subcube~$\subcube{\phi}{U}$.
\end{proof}

\paragraph{P-LCP-induced USOs.} We now formally define the USO that is
induced by a nondegenerate P-LCP instance $\PLCP(M,q)$
\cite{StiWat:Digraph-models}.
For $v\in\{0,1\}^n$, let $B(v):=\{j\in[n] : v_j=1\}$ be the canonical
``set representation'' of~$v$. Then the unique-sink orientation $\Phi$ induced
by $\PLCP(M,q)$ is
  \begin{equation}
    \label{eq:usodef}
  v \toPhi v \oplus i  \quad :\Leftrightarrow \quad
  (A_{B(v)}^{-1}q)_i < 0.
  \end{equation}
  Recall that $A_B$ is the (invertible) $n\times n$ matrix whose $i$th
  column is the $i$th column of~$-M$ if $i\in B$, and the $i$th column
  of the $n\times n$ identity matrix~$I_n$ otherwise.

In this way, we have reduced the P-LCP to the problem of finding
the sink in an implicitly given USO. Access to the USO is gained
through a \emph{vertex evaluation oracle} that for a given vertex
$v\in\{0,1\}^n$ returns the orientations $\Phi(v)\in\{-,+\}^n$ of all
incident edges.  In this setting, the aim is to find the sink of the
USO so that the number of queries to the oracle would be bounded by
a polynomial in the dimension~$n$. In the case of a P-LCP-induced
USO, a (strongly) polynomial-time implementation of the vertex
evaluation oracle is immediate from~\eqref{eq:usodef}.  From the
sink of~$\Phi$, we can reconstruct solution vectors $w$ and~$z$ as
in~\eqref{eq:lcpsol}. The fact that the $n$-cube orientation defined
in this way is indeed a unique-sink orientation was proved by Stickney
and Watson~\cite{StiWat:Digraph-models}.

Not every USO is P-LCP-induced; for $n=3$, the P-LCP-induced USOs are
characterized in~\cite{StiWat:Digraph-models}, but for $n\geq 4$,
we have only a necessary condition. This condition was originally
proved by Holt and Klee for polytope graphs oriented by linear
functions~\cite{HolKle:A-proof}.

\begin{theorem}[\cite{GarMorRus:Unique}]
  Every P-LCP-induced USO of the $n$-cube satisfies the
  \emph{Holt-Klee condition}, meaning that there are $n$ directed
  paths from the source to the sink with pairwise disjoint sets of
  interior vertices.\label{thm:hk}
  \qed
\end{theorem}

\section{Pivot Rules}
\label{sec:pivot}

In the USO setting resulting from P-LCPs, simple principal pivoting
can be interpreted as follows: start from any vertex, and then proceed
along outgoing edges until the global sink (and thus the solution to
the P-LCP) is reached. A \emph{pivot rule} {\sc r} determines which
outgoing edge to choose if there is more than one option. Here is the
generic algorithm, parameterized with the rule {\sc r}. It outputs the
unique sink of the given USO~$\Phi$.

\begin{algorithm}~\par\medskip\noindent
$\textsc{SimplePrincipalPivoting}_{\textsc{r}}(\Phi)$
\begin{algorithmic}
\STATE Choose $v\in\{0,1\}^n$
\WHILE {${\cal O}:=\{j\in[n]:v\toPhi v\oplus j\} \neq \emptyset$}
    \STATE choose $i\in{\cal O}$ according to the pivot rule {\sc r}
    \STATE $v:= v\oplus i$
\ENDWHILE
\RETURN $v$
\end{algorithmic}
\label{alg:pivot}
\end{algorithm}

Note that when $\Phi$ contains directed cycles as in Figure
\ref{fig:uso}, this algorithm may enter an infinite loop for certain
rules {\sc r}, but we consider only rules for which this does not
happen. In the USO setting we are restricted to \emph{combinatorial}
rules. These are rules that may access only the orientations of edges
$\{v,v\oplus j\}$ (as given by the vertex evaluation oracle) but not
any numerical values such as $(A_{B(v)}^{-1}q)_j$ that define these
orientations. 

Let us now introduce the combinatorial pivot rules that are
of interest to us. We write them as functions of the set
${\cal O}$ of candidate indices. The first one is Murty's least-index
rule {$\textsc{Murty}$} that simply chooses the smallest candidate. In the cube, this may
be interpreted as a ``greedy'' approach that always traverses the first 
outgoing edge that it finds (in the order of the cube dimensions). 

\bigskip
\hbox to\linewidth{\hfil
\begin{minipage}[t]{.45\linewidth}
\noindent$\textsc{Murty}({\cal O})$
\begin{algorithmic}
\RETURN $\min({\cal O})$
\end{algorithmic}
\end{minipage}
\hfil
\begin{minipage}[t]{.45\linewidth}
\noindent$\textsc{Murty}_{\pi}({\cal O})$
\begin{algorithmic}
\RETURN $\pi(\min(\pi^{-1}({\cal O})))$
\end{algorithmic}
\end{minipage}\hfil}
\bigskip

It is easy to prove by induction on $n$ that this rule leads to a finite
number of iterations, even if the USO contains directed cycles. The rule
has a straightforward generalization, using a fixed permutation $\pi\in
S_n$: choose the index $\pi(i)$ with the smallest $i$ 
such that $v\to v\oplus\pi(i)$.
This just reshuffles the cube dimensions by applying $\pi$, and with
$\pi=\id$, we recover Murty's least-index rule.

The natural randomized variant of $\textsc{Murty}_{\pi}$ is
\textsc{RandomizedMurty}: at the beginning, choose the permutation~$\pi$
uniformly at random from the set~$S_n$ of all permutations, and then
use {$\textsc{Murty}_{\pi}$} throughout.

Finally, \textsc{RandomEdge} stands for the plain random walk: in each
iteration, simply choose a random candidate.

\bigskip
\hbox to\linewidth{\hfil
\begin{minipage}[t]{.51\linewidth}
\noindent$\textsc{RandomizedMurty}({\cal O})$
\begin{algorithmic}
\IF {called for the first time}
   \STATE choose $\pi\in S_n$ uniformly at random
\ENDIF
\RETURN $\textsc{Murty}_{\pi}({\cal O})$
\end{algorithmic}
\end{minipage}
\hfil
\begin{minipage}[t]{.43\linewidth}
\noindent$\textsc{RandomEdge}({\cal O})$
\begin{algorithmic}
  \STATE choose $i\in {\cal O}$ uniformly at random
  \RETURN $i$
\end{algorithmic}
\end{minipage}\hfil}
\bigskip

Unlike the previous rules, \textsc{RandomEdge} may lead to cycling in
$\textsc{SimplePrincipal}\-\textsc{Pivoting}_{\textsc{r}}$, but the
algorithm still terminates with probability $1$, and with an
expected number of at most $(n+1)n^n$ iterations. This is a simple consequence of the
fact that there is always a short directed path to the sink.

\begin{lemma}[{\cite[Proposition~5]{StiWat:Digraph-models}}]
\label{lem:hamming}
Let $t\in\{0,1\}^n$ be the global sink of an $n$-cube USO~$\Phi$, and let
$v\in\{0,1\}^n$ be any vertex. If $k$ is the Hamming distance between
$v$ and~$t$, then $\Phi$~contains a directed path of length~$k$ from
$v$ to~$t$.
\qed
\end{lemma}

\section{The Morris orientations}
Morris proved that under $\textsc{r}=\textsc{RandomEdge}$,
Algorithm~\ref{alg:pivot} may be forced to perform
an expected \emph{exponential} number of iterations 
\cite{Mor:Randomized-pivot}. More precisely, at least
${((n-1)/2)!}$ iterations are required on average to find the sink
of the $n$-cube
USO generated (as described in Section \ref{sec:uso})
by the $\LCP(M,q)$, with
\begin{equation}
\label{eq:MorrisMq}
M=
\begin{pmatrix}
1&2&0&\dotso&0&0&0\\
0&1&2&\dotso&0&0&0\\
0&0&1&\dotso&0&0&0\\
\hdotsfor7\\
0&0&0&\dotso&1&2&0\\
0&0&0&\dotso&0&1&2\\
2&0&0&\dotso&0&0&1\\
\end{pmatrix},\qquad
q=
\begin{pmatrix}
-1\\-1\\-1\\\vdots\\-1\\-1\\-1
\end{pmatrix}.
\end{equation}
Here, $n$ must be an odd integer in order for $M$ to be a P-matrix.

To prove this result Morris first extracted the relevant structure of
the underlying USO, and then showed that on this USO, {\sc
  RandomEdge} runs in cycles for a long time before it finally
reaches the global sink. It was left as an open problem to determine
the expected performance of {\sc RandomizedMurty}, the other natural
randomized rule, on this particular USO.

We solve this open problem by showing that for \emph{any} permutation
$\pi$, $\textsc{Murty}_{\pi}$ incurs only $O(n^2)$ iterations on the
$n$-dimensional Morris orientation. In particular, this bound then also
holds for {\sc RandomizedMurty}.

\begin{theorem}
\label{thm:general}
For any $\pi\in S_n$, Algorithm \ref{alg:pivot} with
$\textsc{r}=\textsc{Murty}_{\pi}$ finds the sink of the $n$-dimensional
Morris orientation after at most $2n^2-(5n-3)/2$ iterations, from any start
vertex.
\end{theorem}

Computational experiments suggest that the worst running time is
attained for the identity permutation if the algorithm starts from $0$
(the global source). Therefore we analyze this case thoroughly and
compute the precise number of iterations.

\begin{theorem}
\label{thm:id}
Algorithm \ref{alg:pivot} with $\textsc{r}=\textsc{Murty}_{\id}$ finds the
sink of the $n$-dimensional Morris orientation after $(n^2+1)/2$
iterations, starting from $0$.
\end{theorem}

Experimental data suggest that this is an upper bound on the number of
iterations for any permutation $\pi$ and start vertex $0$. Allowing
start vertices different from $0$, we get slightly higher iteration
numbers, but $\pi=\id$ continues to be the worst permutation. In view
of this, we conjecture that $n^2/2+O(n)$ iterations actually suffice
for all $\pi$ and all start vertices. This would mean that 
the bound in Theorem \ref{thm:general} is still off by a factor of~$4$.

\paragraph{The combinatorial structure.}
The $n$-cube USO resulting from (\ref{eq:MorrisMq}) can be described in
purely combinatorial terms \cite[Lemma 1]{Mor:Randomized-pivot}. Here
we make its structure even more transparent by exhibiting a
\emph{finite state transducer} (finite automaton with output) with
just two states that can be used to describe the USO, see Figure
\ref{fig:MorrisFST}. For the remainder of this section, we fix $n$ to
be an odd integer and we are considering the $n$-dimensional
Morris orientation~$\phi$.

\begin{figure}[htb]
\begin{center}
\includegraphics{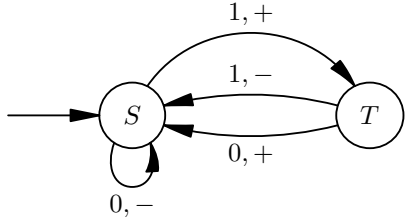}
\end{center}
\caption{The finite state transducer that generates the Morris
  orientations. Each transition is labeled with an input symbol (bit) and
an output symbol (sign)\label{fig:MorrisFST}}
\end{figure}

The orientation is then determined for each vertex
$v=(v_1,v_2,\dotsc,v_n)\in\{0,1\}^n$ as follows. If
$v_1=v_2=\dotsb=v_n=1$, then $v$~is the global sink.  Otherwise
choose some~$i$ in~$[n]$ so that $v_i=0$ and consider the bit
string $v_{i-1}v_{i-2}\dots v_1v_nv_{n-1}\dots v_i$ as the input
string to the transducer. In other words, the transducer is reading the
input from right to left, starting immediately to the left of some $0$. 
The output string
$\phi(v)_{i-1}\phi(v)_{i-2}\dots \phi(v)_1\phi(v)_n\phi(v)_{n-1}\dots
\phi(v)_i$ then determines the orientation of each edge incident
to~$v$. The choice of~$i$ does not matter, because after reading any zero,
the transducer is in the state~$S$.

For instance, let $n=5$ and let $v=(1,0,1,1,0)$. Then the transducer
takes $v_4=1$, outputs $+$ and switches to state~$T$; reads $v_3=1$,
outputs $-$ and switches to~$S$; reads $v_2=0$, outputs $-$ and stays
in~$S$; reads $v_1=1$, outputs $+$ and switches to~$T$; and finally reads
$v_5=0$, outputs~$+$ and switches to~$S$. The resulting configuration for~$v$ 
is
$\left(\begin{smallmatrix}
1&0&1&1&0\\
+&-&-&+&+\\
\end{smallmatrix}\right)$.

Figure \ref{fig:uso} depicts the Morris orientation for $n=3$. It
can easily be checked that the above procedural definition of the 
orientation is equivalent to the one given in \cite[Lemma
1]{Mor:Randomized-pivot}. In particular, we obtain a USO. It has 
the remarkable symmetry that the set of configurations is closed under
cyclic shifts.

\paragraph{Pivoting.} From the transducer
we can easily derive the sign changes in the
configuration~$\left(\begin{smallmatrix}v\\\phi(v)\\\end{smallmatrix}\right)$
that are caused
by a pivot step $v\to u:=v\oplus\{i\}$. We always have $\phi(v)_i=-$
and $\phi(u)_i=+$. 

If $v_i=1$, we get $u_i=0$. In processing $v$ and $u$, the transducer
then performs exactly the same steps, except that at the $i$th
coordinate different transitions are used to get from~$T$ to~$S$. But
this implies $\phi(u)_j=\phi(v)_j$ for $j\neq i$. Here is an example for
such a pivot step (with $i=3$; affected signs are indicated):
\begin{equation}
\label{eq:pivot1}
\left(\begin{smallmatrix}
1&0&1&1&0\\
+&-&-&+&+\\
\end{smallmatrix}\right) \to 
\left(\begin{smallmatrix}
1&0&0&1&0\\
+&-&\oplus&+&+\\
\end{smallmatrix}\right)
\end{equation}

If $v_i=0$, we get $u_i=1$, meaning that at the $i$th coordinate the
transducer stays in~$S$ for~$v$ but switches to~$T$ for~$u$ (assuming
that $u$~is not the sink). This implies that signs change at all
coordinates $i-1$, $i-2$ down to (and including) the next
coordinate~$k$ (wrap around possible) for which $v_k=u_k=0$. In our
example we have such a step for $i=2$:
\begin{equation}
\label{eq:pivot2}
\left(\begin{smallmatrix}
1&0&1&1&0\\
+&-&-&+&+\\
\end{smallmatrix}\right) \to 
\left(\begin{smallmatrix}
1&1&1&1&0\\
\ominus&\oplus&-&+&\ominus\\
\end{smallmatrix}\right)
\end{equation}
In both $v$ and $u$, the signs in the block of $1$'s at indices
$k+1,\dotsc,i-1$ alternate.

\paragraph{Levels.}
It will be useful to define the following function on the vertices of
the cube. Let $\ell(v)$ be the number of coordinates where $\binom{0}{-}$
appears. Formally,
\[\ell(v) := |\{i: v_i=0 \text{ and } \phi(v)_i = {-} \}|.\]
The number $\ell(v)$ is called the \emph{level} of~$v$. From the two
types of pivots, it is easy to see that the value of~$\ell$ does
not increase along any directed
edge~\cite[Lemma~3]{Wes:Notes-on-Morris}. Indeed, if $u$~is an
out-neighbor of~$v$, then either $\ell(u)=\ell(v)$ (this happens
in every pivot of type \eqref{eq:pivot1}, and in pivots of type
\eqref{eq:pivot2} with an odd block of $1$'s at indices $k+1,\dotsc,i-1$),
or $\ell$
decreases to the next possible value: $\ell(u)=\ell(v)-2$, or
$\ell(u)=0$ if $\ell(v)=1$. In particular, the values of~$\ell$ lie in
the set $\{n,n-2,\dotsc,3,1,0\}$.  The sink is the only vertex at
level~0.

Let us make a small digression and briefly explain why
\textsc{RandomEdge} is slow on~$\Phi$. Let $L(v):=|\{i:v_i=1 \text{
  and } \phi(v)_i = {-} \}|$. So the outdegree of~$v$ is
$\ell(v)+L(v)$. Now let $v$ be a vertex at level~$1$. In pivots of
type \eqref{eq:pivot1}, $L$~decreases by one, whereas in pivots of
type \eqref{eq:pivot2} that do not reach the sink, $L$~increases by
one. The latter occurs if and only if \textsc{RandomEdge} pivots at the unique
$\binom{0}{-}$, which happens with probability $1/(L(v)+1)$.

At level~$1$, we can thus interpret \textsc{RandomEdge} as a random
walk on the integers, where $k\rightarrow k+1$ with probability
$1/(k+1)$, and $k\rightarrow k-1$ with probability $k/(k+1)$.  Therefore
the walk is strongly biased towards~$0$, but in order to reach a
neighbor~$u$ of the sink, we need to get to $k=L(u)=(n-1)/2$. This
takes exponentially long in expectation~\cite{Mor:Randomized-pivot}.

\paragraph{A quadratic bound for $\textsc{Murty}_{\pi}$.}
To prove Theorem~\ref{thm:general} we first derive a (somewhat
surprising) more general result: for \emph{any} pivot rule, starting
from any vertex $v$ with $v_k=0$ for some $k$, it takes $O(n^2)$
iterations to reach a vertex $u$ with $u_k=1$. This can be viewed as a
statement about the USO induced by the $n$-dimensional Morris
orientation on the cube facet $\{v:v_k=0\}$. The statement is that in
this induced USO, all directed paths are short (in particular, the
induced USO is acyclic).

\begin{lemma}
\label{lem:poten}
Let $v\in\{0,1\}^n$ be a vertex, and let $k$ be any index for which
$v_{k}=0$. Starting from~$v$, Algorithm \ref{alg:pivot} (with an arbitrary
pivot rule {\sc r}) reaches a vertex~$u$ satisfying $u_{k}=1$
after at most $\binom{n+1}{2}$ iterations.
\end{lemma}

\begin{proof}
  We employ a nonnegative potential that decreases with every
  pivot step. By a cyclic shift of coordinates we may assume that
  $k=1$, and thus $v_1=0$.

  Let us define
  \begin{align*}    
    N_0(v) &:=  \{j\in[n]: v_j = 0 \text{ and } \phi(v)_j={-}\},\\
    N_1(v) &:= \{j\in[n]: v_j = 1 \text{ and } \phi(v)_j={-}\}.
  \end{align*}
  The \emph{potential} of $v$ is the number
  \begin{equation*}
    p(v) = |N_1(v)| + \sum_{j\in N_0(v)} j.
  \end{equation*}

  Now we will observe how the potential changes during a pivot step
  $v\to u=v\oplus\{i\}$. We always have $\phi(v)_i=-$, $\phi(u)_i=+$.
  There are two cases. If $v_i=1$, as in \eqref{eq:pivot1}, the
  configurations of $v$ and $u$ differ only at index $i$. Hence
  $N_0(u) = N_0(v)$ and $N_1(u)=N_1(v)\setminus\{i\}$. Therefore
  $p(u)=p(v)-1$.

  In the other case we have $v_i=0$, as in \eqref{eq:pivot2}.  If
  $i=1$, we are done, since then $u_1=1$. Otherwise, $v$ is of the
  form 
  \[ 
  v=(v_1,\ldots,\underbrace{v_{j}}_{0},\underbrace{v_{j+1},\ldots,
    v_{i-1}}_{1,\ldots, 1},\underbrace{v_i}_{0},v_{i+1},\ldots,v_n),
  \]
  with $1\leq j<i$. If $i-j$ is odd, there is an even number of $1$'s
  between $v_{j}$ and $v_i$, implying that $|N_1(u)|=|N_1(v)|$ and
  $N_0(u)=N_0(v)\setminus\{j,i\}$. We then have $p(u)<p(v)$. If
  $i-j$ is even, there is an odd number of $1$'s between $v_{j}$
  and $v_i$, implying that $|N_1(u)|=|N_1(v)|+1$ and
  $N_0(u)=N_0(v)\cup\{j\}\setminus\{i\}$. Since $i-j\geq 2$ in
  this case, we get $p(u)\leq p(v)-1$.

  To summarize, the potential decreases by at least 1 in every pivot
  step $v:=v\oplus i$, and as long as $v$ is not the sink, we have
  $p(v)\geq 1$. The highest potential is attained by the source,
  $p(0)=\binom{n+1}{2}$. Therefore after at most $\binom{n+1}{2}$
  steps we have $v_1=1$.
\end{proof}

If $v$ happens to be at level~1 already, we obtain a better bound,
since $|N_0(v)|=1$ in this case. In fact, the vertex with largest
potential in level~1 such that $v_1=0$ is $v=\left(\begin{smallmatrix}
    0&1&1&1&1&\dotso&1&1&1&0\\
    +&+&-&+&-&\dotso&+&-&+&-
\end{smallmatrix}\right)$, and its potential is $p(v)=3(n-1)/2$.

\begin{cor}
\label{cor:poten}
Let $v\in\{0,1\}^n$ be a vertex at level~1, and let $k$ be any index for which
$v_{k}=0$. Starting from~$v$, Algorithm \ref{alg:pivot} (with arbitrary
pivot rule {\sc r}) reaches a vertex~$u$ satisfying $u_{k}=1$
after at most $3(n-1)/2$ iterations.
\qed
\end{cor}

Now $\textsc{Murty}_{\pi}$ comes in: we will apply either
Lemma~\ref{lem:poten} or Corollary~\ref{cor:poten} to vertices $v$
with $v_{\pi(i)}=0$ and $v_{\pi(i+1)}=\dotsb=v_{\pi(n)}=1$, for some
$i$. The next lemma shows that the $1$'s at coordinates
$\pi(i+1),\ldots,\pi(n)$ will stay.

\begin{lemma}
\label{lem:visonly}
Let $i\in[n]$. After Algorithm \ref{alg:pivot} with $\textsc
{r}=\textsc{Murty}_{\pi}$ visits a vertex~$v$ with
$v_{\pi(i+1)}=\dotsb=v_{\pi(n)}=1$, it will visit only vertices~$u$
satisfying $u_{\pi(i+1)}=\dotsb=u_{\pi(n)}=1$.
\end{lemma}

\begin{proof}
  At every iteration, the algorithm replaces the current vertex $v$
  with $v\oplus\pi(j)$, where $j$ is the smallest index such that
  $v\to v\oplus\pi(j)$. It follows that coordinates
  $\pi(i+1),\ldots,\pi(n)$ will be touched only if $v\oplus\pi(j)\to
  v$ for all $j\leq i$.  When this holds for the first
  time, we have
  \begin{alignat*}{2}
    v_{\pi(j)} &= 1, &\quad&\text{for $j>i$},\\
    \phi(v)_{\pi(j)} &= +, &&\text{for $j\leq i$}.
  \end{alignat*}
  
The unique-completion property (Lemma \ref{lem:ucp}) then implies
that we have already reached the sink.
\end{proof}

Now we can put it all together. 

\begin{proof}[Proof of Theorem~\ref{thm:general}]
  Let $v$ be the start vertex, and let $i$ be the largest index such
  that $v_{\pi(i)}=0$ (if no such index exists, $v$ is the sink, and
  we are done). Lemmas~\ref{lem:poten} and \ref{lem:visonly} show that
  after at most $\binom{n+1}{2}-1$ iterations, we additionally have
  $\phi(v)_{\pi(1)}=\phi(v)_{\pi(2)}=\cdots=\phi(v)_{\pi(i-1)}=+$ (since
  the next pivot step flips $v_{\pi(i)}$).  At this point $v$~is at
  level~1, since there is a unique~$\binom{0}{-}$: at
  coordinate~$\pi(i)$. After the flip at coordinate~$\pi(i)$, we repeatedly apply
  Corollary~\ref{cor:poten} together with Lemma \ref{lem:visonly} to
  also produce $1$'s at coordinates
  $\pi({i-1}),\ldots,\pi({i-2}),\ldots,\pi(1)$, at which point we have
  reached the sink. This takes at most $3(i-1)(n-1)/2$ more
  iterations, summing up to a total of
\[\binom{n+1}{2} + 3(i-1)(n-1)/2 \leq \binom{n+1}{2} + 3(n-1)(n-1)/2
  = 2n^2-(5n-3)/2.
\]  
\end{proof}

A better bound in the case of the identity permutation can be
achieved by thoroughly examining the run of the algorithm, as we do next.

\paragraph{The identity permutation.}
Let us now go on to prove Theorem~\ref{thm:id}. For this, we first
identify certain ``milestone'' vertices that are visited by Algorithm
\ref{alg:pivot} with $\textsc{r}=\textsc{Murty}_{\id}$, and then we
count the number of iterations to get from one milestone to the next.
Let us define
\[v^i := (\underbrace{0,1,0,1\ldots,0,1}_{2i},0,0,\ldots,0),
\quad i=0,\ldots,(n-1)/2.\]

\begin{lemma}
\label{lem:l4}
If Algorithm
\ref{alg:pivot} with $\textsc{r}=\textsc{Murty}_{\id}$ starts
at the vertex $v^i$, $0\leq i<(n-1)/2$, it gets to the vertex $v^{i+1}$
in $4i+3$ iterations.
\end{lemma}

\begin{proof}
The algorithm starts at 
$\left(\begin{smallmatrix}
0&1&0&1&\cdots&0&1&~&0&0&\cdots\\
+&+&+&+&\cdots&+&+&~&-&-&\cdots\\
\end{smallmatrix}\right)$, and then proceeds by the rules of pivoting  
\eqref{eq:pivot1} and \eqref{eq:pivot2} in four phases (recall that 
$\textsc{Murty}_{id}$ always pivots on the leftmost~$-$):

\begin{itemize}
\item It pivots on the coordinates $2i+1$, $2i-1$,~\dots,~$1$
and thus reaches the vertex\\
$\left(\begin{smallmatrix}
1&1&1&1&\cdots&1&1&~&1&0&\cdots\\
+&-&+&-&\cdots&+&-&~&+&-&\cdots\\
\end{smallmatrix}\right)$ after $i+1$~iterations.

\item It pivots on the coordinates $2$, $4$, $6$,~\dots,~$2i$, $2i+2$
and thus reaches the vertex\\
$\left(\begin{smallmatrix}
1&0&1&0&\cdots&1&0&~&1&1&\cdots\\
+&+&+&+&\cdots&+&-&~&-&+&\cdots\\
\end{smallmatrix}\right)$
after $i+1$~iterations; 

\item It pivots on the coordinates $2i$, $2i-2$,~\dots,~$2$
and thus reaches the vertex\\
$\left(\begin{smallmatrix}
1&1&1&1&\cdots&1&1&~&1&1&\cdots\\
-&+&-&+&\cdots&-&+&~&-&+&\cdots\\
\end{smallmatrix}\right)$
after $i$~iterations;

\item It pivots on the coordinates $1$, $3$, $5$,~\dots,~$2i+1$
and thus reaches the vertex\\
$\left(\begin{smallmatrix}
0&1&0&1&\cdots&0&1&~&0&1&\cdots\\
+&+&+&+&\cdots&+&+&~&+&+&\cdots\\
\end{smallmatrix}\right) = v^{i+1}$
after $i+1$~iterations.
\end{itemize}

Therefore it takes $4i+3$ iterations to get from $v^i$ to $v^{i+1}$.
\end{proof}

The previous lemma allows us to count the number of iterations from 
$0=v^0$ to $v^{(n-1)/2}$. The following lemma takes care of the remaining
iterations. 

\begin{lemma}
\label{lem:l6}
If Algorithm
\ref{alg:pivot} with $\textsc{r}=\textsc{Murty}_{\id}$ starts
at the vertex $v^{(n-1)/2}$, it gets to the global sink in $(n+1)/2$
iterations.
\end{lemma}

\begin{proof}
The proof is similar to the proof of Lemma \ref{lem:l4}, except that here
only the first phase takes place: if the algorithm starts at the vertex
$\left(\begin{smallmatrix}
0&1&0&1&\cdots&0&1&~&0\\
+&+&+&+&\cdots&+&+&~&-\\
\end{smallmatrix}\right)$,
it pivots on coordinates $n$, $n-2$,~\dots,~$1$ and thus 
reaches the sink after $(n+1)/2$ iterations.
\end{proof}

\begin{proof}[Proof of Theorem~\ref{thm:id}]
By Lemmas~\ref{lem:l4} and \ref{lem:l6}, the overall number of iterations
from the source to the sink is
\[\sum_{i=0}^{(n-3)/2}(4i+3) + \frac{n+1}{2} = \frac{n^2+1}{2}.\qedhere\]
\end{proof}

\section{K-matrix LCP}

In this section we examine \emph{K-LCPs}, the linear complementarity
problems $\LCP(M,q)$ where $M$~is a K-matrix.
\begin{definition}
  A \emph{K-matrix} is a P-matrix so that all 
off-diagonal entries are non-positive.\footnote{Another
    common name in the literature for a K-matrix is \emph{Minkowski
      matrix}.}
\end{definition}

We introduce two simple combinatorial conditions on
unique-sink orientations and prove that one of them implies that
all directed paths from~$0$ are short, and the other one implies
that all directed paths in the orientation are short. We show that
K-LCP orientations satisfy these conditions and conclude that a simple
principal pivoting method with an arbitrary pivot rule solves a K-LCP in
linearly many iterations, starting from any vertex.

First, the \emph{uniform orientation} is the USO in which $v\to v\oplus
i$ if and only if $v_i=0$ (in other words, all edges are oriented ``from
0 to~1''). 

\begin{definition}
A unique-sink orientation~$\Phi$ is \emph{2-up-uniform} if
whenever $U=\{u\oplus I: I\subseteq\{i,j\}\}$  is
such that $u_i=u_j=0$ and $u$~is the source of~$U$, then the orientation
of the subcube~$\subcube{\phi}{U}$ is uniform (see Figure~\ref{fig:2uu}).
\end{definition}

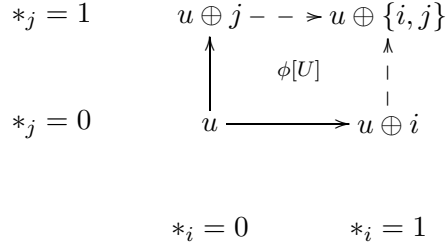
\begin{figure}[htb]
\begin{center}
$\xymatrix{
\ast_j=1&u\oplus j \ar@{-->}[r]\ar@{}[dr]|{\subcube{\phi}{U}}& u\oplus\{i,j\}\\
\ast_j=0&u\ar[u]\ar[r]& u\oplus i\ar@{-->}[u]\\
&\ast_i=0&\ast_i=1}$
\end{center}
\caption{2-up-uniformity: The orientations of the solid edges imply
the orientations of the dashed edges. The top row contains vertices with
$j$th entry equal to~1 and the bottom row vertices with $j$th entry equal
to~0; the left column contains vertices with $i$th entry equal to~0, and
the right column contains vertices with $i$th entry equal to~1.}
\label{fig:2uu}
\end{figure}

A \emph{K-USO} is a unique-sink orientation that arises (as described in
Section~\ref{sec:uso}) from the $\LCP(M,q)$ with some K-matrix~$M$
and some right-hand side~$q$.

\begin{prop}
\label{prop:k-2-up}
Every K-USO is 2-up-uniform.
\end{prop}

\begin{proof}
The essential fact we use is that a subcube of a K-USO
is also a K-USO, as was observed already by Stickney and
Watson~\cite[Lemma~1]{StiWat:Digraph-models}.  Thus it suffices to
prove 2-up-uniformity for K-LCPs of dimension~2. Hence suppose
that $M=\left(\begin{smallmatrix}a&b\\c&d\\\end{smallmatrix}\right)$
is a $2\times2$ K-matrix, so $a,d>0$ and $b,c\le0$.  Then
$M^{-1}=\left(\begin{smallmatrix}d&-b\\-c&a\\\end{smallmatrix}\right)/
{(ad-bc)} \ge0$ and nonsingular.  Now 0~is the source of the orientation
induced by the~$\LCP(M,q)$ if and only if $q<0$.  Hence $-M^{-1}q>0$,
which proves uniformity.
\end{proof}

One particular nice property of a 2-up-uniform orientation
is that all paths from the vertex~0 to the sink~$t$ have the shortest
possible length, equal to the Hamming distance $|t|$ of~$t$ from~$0$.
To prove this statement, we use the following lemma.

\begin{lemma}
\label{lem:crufa-up}
Let $\Phi$ be a 2-up-uniform USO and let $v$ be a vertex such that
$v_i=1$ and $v\oplus i\toPhi v$. If $v_j=0$ and $u = v\oplus j$ is an out-neighbor
of $v$, then $u\oplus i\toPhi u$.
\end{lemma}

\begin{proof}
Suppose for the sake of contradiction that $v_i=1$, $v_j=0$ and
$v\oplus i\to v$, and that $u = v\oplus j$, $v\to u$ and $u\to
u\oplus i$.  Let $U:=\{u\oplus I:I\subseteq\{i,j\}\}$ and consider
the subcube~$\subcube{\phi}{U}$ (see Figure~\ref{fig:uni-proof1}, left): we
observe that $v\oplus i\to u\oplus i$ because otherwise the subcube would
not contain a sink. But then it follows from 2-up-uniformity of~$\Phi$
that $u\oplus i\to u$, a contradiction.
\end{proof}

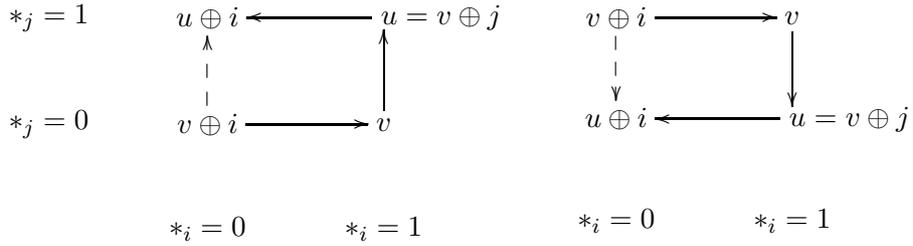
\begin{figure}[htb]
\begin{center}
$\xymatrix{
\ast_j=1&u\oplus i& *[r]{\ u=v\oplus j}\ar[l]\\
\ast_j=0&v\oplus i\ar[r]\ar@{-->}[u]& v\ar[u]\\
&\ast_i=0&\ast_i=1}$
\qquad 
$\xymatrix{
v\oplus i\ar[r]\ar@{-->}[d]& v\ar[d]\\
u\oplus i&*[r]{\ u=v\oplus j}\ar[l]\\
\ast_i=0&\ast_i=1}$
\end{center}
\caption{Illustrating the proof of Lemma~\ref{lem:crufa-up} (left)
and Lemma~\ref{lem:crufa} (right).}
\label{fig:uni-proof1}
\end{figure}

\begin{prop}
\label{prop:2uu-short}
Let $\Phi$ be a 2-up-uniform USO and let
$t$~be its sink. Then any directed path from~$0$ to~$t$ has length~$|t|$.
\end{prop}

\begin{proof}
It follows from Lemma~\ref{lem:crufa-up} by induction that if $u$~is a vertex on any
directed path starting from~$0$, and if $u_i=1$, then $u\oplus i\to u$.
Hence all edges on a directed path from~$0$ to~$t$ are oriented from~0
to~1, and therefore the length of a directed path from~$0$ to~$t$
is~$|t|$.
\end{proof}

As a direct consequence of Propositions \ref{prop:k-2-up}
and~\ref{prop:2uu-short} we get the following theorem.

\begin{theorem}
Every directed path from~0 to the sink~$t$ of a K-USO has length
${|t|\le n}$.
\qed
\end{theorem}

What is the actual strength of this theorem? We know that from any
vertex of a unique-sink orientation there \emph{exists} a path to the sink
of length at most~$n$ (Lemma \ref{lem:hamming}). It has also long been
known how to find such a path from the vertex~0 in the case of a
K-USO~\cite{Cha:A-special,Sai:A-note}. The
novelty is that \emph{any} directed path starting from~0 reaches the
sink after the least possible number of iterations. Hence a simple
principal pivoting method with an \emph{arbitrary} pivot
rule finds the solution to a K-LCP in at most $n$~iterations.

In view of this result, it is natural to ask whether it also holds if
the path starts in some vertex different from~0. Imposing
2-up-uniformity is insufficient, as we can observe by looking at the
Morris orientations. If we swap 0's and 1's in any Morris orientation,
we get a 2-up-uniform orientation. Hence all
directed paths from~0 to the sink are short; indeed, 0~is the sink of
the orientation after swapping, so all such paths have length~0.
Nevertheless, long directed paths (and even directed cycles) do exist in 
this orientation. Therefore we introduce the following stronger
combinatorial property of USOs, which turns out to be satisfied by 
K-USOs as well.

\begin{definition}
A unique-sink orientation $\Phi$ is \emph{2-uniform} if it
is 2-up-uniform and the orientation~$\Phi^{([n])}$ constructed
from~$\Phi$ by reversing all edges is also 2-up-uniform.
\end{definition}

\begin{prop}
Every K-USO is 2-uniform.
\label{prop:Kuni}
\end{prop}

\begin{proof}
If $\Phi$ is induced by the $\LCP(M,q)$ for some K-matrix~$M$, then
$\Phi^{([n])}$~is induced by the $\LCP(M,-q)$, hence it is also a K-USO and
therefore it is 2-up-uniform.
\end{proof}

Surprisingly, the simple property of being 2-uniform enforces
a lot of structure on the whole orientation, as we show next.

\begin{theorem}
\label{thm:2lu}
The length of every directed path in a 2-uniform USO is at
most~$2n$.
\end{theorem}

Before we prove the theorem, let us point out an important corollary.

\begin{cor}\label{cor:2n}
Algorithm \ref{alg:pivot} with an arbitrary pivot
rule~\textsc{r}, starting at an arbitrary vertex of the corresponding USO,
finds the solution to  an $n$-dimensional K-LCP in at most $2n$~iterations.
\qed
\end{cor}

As a consequence, we get a result for a larger class of LCPs. If
$M$~is a \emph{principal pivotal transform} \cite[Section
2.3]{CotPanSto:LCP} of a K-matrix~$M'$, then $M$~is a P-matrix
\cite[Theorem 4.1.3]{CotPanSto:LCP}, and the USO $\Phi$ induced by
the $\LCP(M,q)$ is isomorphic to the USO~$\Phi'$ induced by the
$\LCP(M',q')$ for suitable $q'$.  It follows that
Corollary~\ref{cor:2n} also applies to the $\LCP(M,q)$, even though
$M$~is not necessarily a K-matrix itself.

Our proof of Theorem~\ref{thm:2lu} is based on the following crucial fact,
which extends Lemma~\ref{lem:crufa-up}.
Informally, it asserts that once we have a
$\binom{1}{+}$ in some coordinate, we will always have a
$\binom{1}{+}$.

\begin{lemma}
\label{lem:crufa}
Let $\Phi$ be a 2-uniform USO and let $v$ be a vertex such that
$v_i=1$ and $v\oplus i\toPhi v$. If $u = v\oplus j$ is an out-neighbor
of $v$, then $u\oplus i\toPhi u$.
\end{lemma}

\begin{proof}
If $v_j=0$, the claim follows from Lemma~\ref{lem:crufa-up}.  So let us
suppose that we have $v_i=1$, $v_j=1$ and $v\oplus i\to v$, and that
$u = v\oplus j$, $v\to u$ and $u\to u\oplus i$.  Let $U:=\{u\oplus
I:I\subseteq\{i,j\}\}$.  We observe that $v\oplus i\to u\oplus i$
so that the subcube~$\subcube{\phi}{U}$ contains a
sink. But then it follows from 2-up-uniformity of~$\Phi^{([n])}$ that $v\to v\oplus i$,
a contradiction (see Figure~\ref{fig:uni-proof1}, right).
\end{proof}

\begin{proof}[Proof of Theorem~\ref{thm:2lu}]
  Let $i$ be a fixed coordinate. Suppose in the considered directed
  path there is an edge $v\oplus i\to v$ such that $v_i=1$. As a
  consequence of Lemma~\ref{lem:crufa}, all vertices $u$ on a directed
  path starting at $v$ satisfy $u_i=1$.  It follows that in any
  directed path no more than two edges appear of the form $v\oplus
  i\to v$ for any fixed~$i$: possibly one with $v_i=0$ and one with
  $v_i=1$. This fact then implies that the length of any directed path
  is at most~$2n$.
\end{proof}

The strength of 2-uniformity can perhaps be explained by its being
equivalent to what we call \emph{local uniformity}. A USO~$\Phi$ is
\emph{locally up-uniform} if for every
$U=\{u\oplus I: I\subseteq
J\}$ such that $u_J=0$ and $u\toPhi u\oplus j$ for 
all $j\in J$, the orientation of the subcube~$\subcube{\phi}{U}$
is uniform. A USO~$\Phi$ is \emph{locally uniform}
if both $\Phi$ and $\Phi^{([n])}$ are locally up-uniform.

\begin{prop}
Let $\Phi$ be a USO.
\begin{itemize}
\item[(i)] $\Phi$ is locally up-uniform if and only if it is 2-up-uniform.
\item[(ii)] $\Phi$ is locally uniform if and only if it is 2-uniform.
\end{itemize}
\label{prop:locuni}
\end{prop}

\begin{proof}
  (ii) is easily implied by (i); and one implication of (i) is trivial.
  The remaining implication is not difficult to prove by induction on the
  dimension of the considered subcube.
\end{proof}

Local up-uniformity allows for a slight improvement in solving
a K-LCP. When following a directed path starting in~0, instead of
traversing one edge at a time we can perform a ``greedy'' principal
pivot: jump straight to the sink of the subcube induced by outgoing
edges. Unlike the general P-matrix case, it cannot lead to cycling. Due
to local up-uniformity, it will never increase the number of iterations
and may sometimes reduce it.

\paragraph{Locally uniform USOs vs.\ K-USOs.} We have shown
(Propositions~\ref{prop:Kuni} and \ref{prop:locuni}) that every K-USO
is 2-uniform (equivalently, locally uniform). It is natural to ask
whether the converse is also true.

The answer to this question is negative. It can be shown that the number
of $n$-dimensional locally uniform USOs is asymptotically much larger
than the number of K-USOs. A lower bound of $2^{\binom{n-1}{\lfloor
(n-1)/2\rfloor}}$ on the number of locally uniform USOs can be derived
by an adaptation of Develin's construction~\cite{Dev:LP-orientations}
of many orientations satisfying the Holt-Klee condition. An upper bound
of~$2^{O(n^3)}$ on the number of P-USOs (and thus also K-USOs) follows
from Develin's proof of his upper bound on the number of LP-realizable
cube orientations.

At present, however, we do not know whether locally uniform P-USOs form
a proper superclass of the class of K-USOs.

\end{document}